\documentclass[final,1p,times]{elsarticle}
\usepackage{tikz}
\usepackage{amsmath,amssymb} 
\usepackage{amsthm}
\usepackage{bm}   
\usepackage{pdfsync}  
\usepackage{subfigure}
\usepackage{color}
\usepackage[english]{babel}

\usepackage[T1]{fontenc}
\usepackage{yfonts}

\theoremstyle{plain}
\newtheorem{theorem}{Theorem}[section]
\newtheorem{proposition}[theorem]{Proposition}

\newtheorem{corollary}[theorem]{Corollary}
\newtheorem{assumption}[theorem]{Assumptions}

\theoremstyle{definition}
\newtheorem{example}[theorem]{Example}
\newtheorem{definition}[theorem]{Definition}
\newtheorem{remark}[theorem]{Remark}

\newtheorem{notation}[theorem]{Notation}

\newcommand{\N}{{\mathbb N}}
\newcommand{\Q}{{\mathbb Q}}

\newcommand{\X}{{\mathbb X}}
\newcommand{\Sc}{{\mathcal S}}
\newcommand{\U}{{\mathcal U}}
\newcommand{\x}{{\mathbf x}}

\newcommand{\aaa}{{\mathbf a}}

\newcommand{\bfalpha}{{\bm \alpha}}

\newcommand{\f}{{\mathbf f}}

\def\LT{\mathop{\rm LT}\nolimits}

\def\Spec{\mathop{\rm Spec}\nolimits}

\def\y1{\mathbf{y}^{(1)}}

\def\To{\longrightarrow}
\def\TTo#1{\mathop{\longrightarrow}\limits^{#1}}

\let\epsilon=\varepsilon
\let\rho=\varrho
\let\phi=\varphi
\let\To=\longrightarrow
\def\TTo#1{\mathop{\longrightarrow}\limits ^{#1}}

\def\longiso{\,\smash{\TTo{\lower 7pt\hbox{$\scriptstyle\sim$}}}\,}

\def\tfrac #1#2{{\textstyle\frac{#1}{#2}}}

\def\cocoa{\mbox{\rm
C\kern-.13em o\kern-.07 em C\kern-.13em o\kern-.15em A}}
\def\apcocoa{\mbox{\rm
A\kern-0.13em p\kern -0.07em C\kern-.13em o\kern-.07 em C\kern-.13em
o\kern-.15em A}}

\begin{document}

\begin{frontmatter}

\title{Hyperplane Sections, Gr\"obner Bases, and Hough Transforms}
\author{Lorenzo Robbiano}
\address{\scriptsize Dipartimento di Matematica, 
\ Universit\`a di Genova, \ Via
Dodecaneso 35,\
I-16146\ Genova, Italy}
\ead{robbiano@dima.unige.it}

\begin{abstract}
Along the lines of~\cite{BR} and ~\cite{R},  the purpose 
of this paper is twofold. 
In the first part we  concentrate on hyperplane sections 
of algebraic schemes,  and present results for determining when 
Gr\"obner bases pass to the quotient and when they can be lifted.  
The main difficulty to overcome is the fact that we 
deal with non-homogeneous ideals.  As a by-product we hint at a 
promising technique for computing implicitization efficiently.

In the second part of the paper we deal with
families of algebraic schemes and the Hough transforms,
in particular we compute their dimension, and show that
in some interesting cases it is zero. Then
we concentrate on their hyperplane sections.
Some results and examples  hint at the possibility of reconstructing external
and internal surfaces of human organs
from the parallel cross-sections obtained by tomography.
\end{abstract}


\end{frontmatter}

\begin{flushright}
{\it these days, even the most pure\\ and abstract mathematics\\ is in danger to be applied}\\
{\rm (Anonymous)}
\end{flushright}

\section*{Introduction}

About 50 years ago the technique of Hough Transforms (HT for short) 
was introduced  with the purpose of recognizing
special curves inside images (see~\cite{H}); 
it has subsequently become widely used and generalized in many ways,
notwithstanding the fact that its range of application is rather limited.   
An extension was presented in~\cite{BR} where the HT was introduced
in the wider context of families of algebraic schemes.  
This paved the way to detecting more complicated objects,
and offered the prospect of using algebraic geometry to help other 
scientists, in particular doctors, in the challenge
of  recognizing and  reconstructing images from various types of
tomography.  

In this paper we commence our investigation by  considering  
hyperplane sections.
It is well-known how to use the {\tt DegRevLex} term ordering to
relate Gr\"obner bases of a homogenous ideal to Gr\"obner bases 
of its quotient modulo a linear form (see Proposition~\ref{genhyper}).  
However, that result makes essential use of
the homogeneity.  Since we have in mind ``inhomogeneous applications"
we develop a theory which 
works in the inhomogenous case.  We prove  
Theorem~\ref{mainhyper} which shows
how Gr\"obner bases pass to the quotient,  and
Theorem~\ref{lifting} which establishes a criterion for lifting Gr\"obner bases.
A consequence of this result is Theorem~\ref{commonliftingGroebner}
which describes a class of instances where  good  liftings of
Gr\"obner bases can be obtained.  A confirmation of its usefulness
comes from some preliminary experiments on computing
implicitizations (see Example~\ref{sectparams}).

The second section concentrates on families of schemes and 
Hough transforms.  
Given a family of algebraic schemes, there is a non-empty Zariski 
open set over which the universal
reduced $\sigma$-Gr\"obner basis $G$ specializes
to the reduced $\sigma$-Gr\"obner basis of the fibers, hence
we get a parametrization of the fibers via the non-constant
coefficients of $G$ (see Proposition~\ref{coefflistuniqe}). The scheme
which parametrizes the fibers is called the $\sigma$-scheme,
and from the theory of Gr\"obner fans (see~\cite{MR}) we deduce that each
family has only a finite set of $\sigma$-schemes (see Corollary~\ref{H-fan}).
Subsection~\ref{HyperSecFam} applies the 
results about hyperplane sections to families of algebraic schemes. 

After recalling our definition of HT 
(see Definition~\ref{Hough}), 
we show how to compute its dimension.  
Proposition~\ref{dimensionoffiber} and some
examples illustrate some 
cases when the HTs are zero-dimensional schemes.
Finally, Example~\ref{recHough}  
shows how  the equation of a surface can be reconstructed from the 
equations of some of its hyperplane sections. 
The important remark here is that  the equations of these  curves
can be obtained obtained using Hough transforms.

Why did we mention other scientists, in particular doctors?
Suppose that special curves have 
been recognized in the tomographic sections of a human organ. 
Our results hint at the possibility of reconstructing 
a surface whose cross-sections coincide with the recognized curves.  
It could be the contour of a vertebra or a kidney.  
However, more difficulties arise, in particular the fact that 
in this context all the data are \textit{inexact}. More investigation is needed
which  is \textit{exactly} what every researcher loves.

\begin{assumption}\label{general}
Throughout the paper we use notation and definitions 
introduced in~\cite{KR1} and~\cite{KR2}. 
Moreover, we always assume that  for
every term ordering~$\sigma$ on
$\mathbb T^n=\mathbb T(x_1,\dots, x_n)$ we have 
${x_1>_\sigma\cdots >_\sigma x_n}$, and
we call  $\hat{\sigma}_i$ (or simply  $\hat{\sigma}$ 
if $i$ is clear from the context) the restriction of $\sigma$ to 
the monoid $\mathbb T(x_1,\dots, x_{i-1},x_{i+1},\dots, x_n)$.
\end{assumption}

My warmest thanks go to J. Abbott, M. C. Beltrametti, 
M. Piana, M. L. Torrente.
With them I had the opportunity to share my ideas, and from 
them I received valuable feed-back.

\tableofcontents

\section{Hyperplane Sections and Gr\"obner Bases}

In this section we conduct our investigation into hyperplane sections,
and establish some results about Gr\"obner bases.  We recall several facts
from the well-known homogeneous case and confront them with new
results for the inhomogenous case.

\subsection{The Homogeneous Case}

\begin{assumption}\label{hyperplane}
Let $P^{\mathstrut}=K[x_1,\dots,x_n]$, 
let $i \in \{1,\dots, n\}$, let $c_1, \dots,c_{i-1},c_{i+1}\dots c_n \in K$,
and let $L=x_i-\ell\ $ be the linear polynomial
with $\ell =  \sum_{j\ne i} c_jx_j$. We identify 
$P/(L)$ with the  ring ${\hat{P} = K[x_1, \dots, 
x_{i-1},x_{i+1},\dots,x_n]}$ via the isomorphism
defined by $\pi_L(x_i) = \ell$, $\pi_L(x_j) = x_j$
for $j\ne i$. 
\end{assumption}

\begin{definition}\label{defdegrevtype}
For $t=x_1^{a_1}x_2^{a_2}\cdots x_n^{a_n}$
we write $a_i = \log_{x_i}(t)$.
For $i\in\{1,\dots,n\}$ we say
that a term ordering~$\sigma$ is of
{\bf {\boldmath $x_i$\unboldmath}-DegRev type}
if it is degree-compatible, and for every pair of terms
$t,t' \in\mathbb T^n$ which satisfy $\deg(t)=\deg(t')$ and 
$\log_{x_i}(t) < \log_{x_i}(t')$, then $t>_\sigma t' $.
\end{definition}

\noindent
The usual $\tt DegRevLex$ ordering is of $x_n$-{\tt DegRev} type, 
and by suitably modifying its definition we see that 
for every $i$ there exist term orderings of $x_i$-{\tt DegRev} type. 
For an in-depth analysis of this topic see~\cite[Section 4.4]{KR2}.

We observe that $\pi_L=\rho\circ \theta$ where 
$\theta : P \To P$ is defined 
by $\theta(x_i) = x_i+\ell$,  $\ \theta(x_j) = x_j$ for $j \ne i$
while 
$\rho : P\To \hat{P}$ is defined by $\rho(x_i) = 0$,
$\ \rho(x_j) = x_j\,$ for $j \ne i$.

\begin{proposition}\label{genhyper}{\bf (Homogeneous 
Hyperplane Sections)}\\
Under Assumptions~\ref{hyperplane}, we let 
let $\sigma$ be a  term 
ordering of  ${x_i}_{\mathstrut}$-{\tt DegRev} type
on $P$, let $I$ be a homogeneous ideal in $P$,
and let $G=\{G_1,\dots, G_s\}$ be the 
reduced $\sigma$-Gr\"obner basis of $\theta(I)$.
Then $\rho(G) = \{\rho(G_1),\dots,\rho(G_s) \} \setminus \{0\}$
is the reduced $\hat{\sigma}$-Gr\"obner basis of~$\pi_L(I)$.
\end{proposition}

\proof  
First of all we observe that~$\theta$ is a graded 
isomorphism, so that $\theta(I)$ is a homogeneous ideal. 
Consequently the claim follows from a classical result in 
Gr\"obner basis theory (see for instance~\cite[Corollary 4.4.18]{KR2}).
\endproof

The following examples illustrate this proposition.

\begin{example}
Let $P=K[x,y,z,w]$, let $I = (z^2-xw,\ x^2y-zw^2)$,  let 
$\ell = 3y+w$, and let $L = z-\ell=z -3y-w$.
We consider the linear change of coordinates $\theta$ 
which sends $z$ to $z+\ell = z+3y+w$, 
and $x$ to $x$, $y$ to $y$, $w$ to $w$.
Then $\theta(I) = (  9y^2 -xw +6yw +w^2 +6yz +2zw+z^2,\ 
  x^2y -{3yw^2}^{\mathstrut} -w^3 -zw^2)$.
Let $\sigma$ be a term ordering of $z$-{\tt DegRev} type.
The reduced $\sigma$-Gr\"obner basis  of $\theta(I)$ is 
 $$
 \begin{array}{lll}
G=\!\!\! &\{\,y^2-\tfrac{1}{9}xw+\tfrac{2}{3}yw+\tfrac{1}{9}w^2+\tfrac{2}{3}yz
+\tfrac{2}{9}wz+\tfrac{1}{9_{\mathstrut}}z^2,\\
&\,\,x^2y-3yw^2-w^3-w^2z,\\
&\,\, x^3w-x^2w^2-3xw^3-9yw^3-3w^4-2x^2wz-9yw^2z
-6w^3z-x^2z^2-3w^2z^2\}
  \end{array}
$$
If we mod out $z$ we get 
 $$
 \begin{array}{lll}
\rho(G)=\!\!\! &\{\,y^2-\tfrac{1}{9}xw+\tfrac{2}{3}yw+\tfrac{1}{9}w^2,\\
&\,\,x^2y-{3yw^2}^{\mathstrut}-w^3,\\
&\,\,x^3w-x^2w^2-3xw^3-9yw^3-3w^4\}
 \end{array}
$$
On the other hand, if $\rho$ is defined by $\rho(x) = x$, $\rho(y) = y$,
$\rho(w) = w$, $\rho(z) = 0$, and we  put $\pi_L = \rho\circ \theta$
we have $\pi_L(I) = (9y^2 -xw +6yw +w^2,\ 
x^2y -{3yw^2}^{\mathstrut} -w^3)$. If we compute 
the $\hat{\sigma}$-reduced Gr\"obner basis 
of  $\pi_L(I)$ we get $\rho(G)$, as prescribed by the proposition.
\end{example}

\begin{example}\label{delete0}
Let $P=K[x_0,x_1,x_2,x_3]$, 
let $F_1 =  x_3^3-x_1x_2x_0$, $F_2=x_2^3-x_1x_3x_0-x_2x_0^2$, 
$F_3 = x_1^2x_2-x_3x_0^2$, 
let $I=(F_1,F_2,F_3)$, and let $\sigma$ be a term ordering
of $x_0$-{\tt DegRev} type.
The reduced Gr\"obner basis of $I$ is $(F_1, F_2, F_3, F_4)$
where $f_4 = x_1^3x_3x_0-x_2^2x_3x_0^2+x_3x_0^4$.
If we mod out $x_0$ we get $\pi_L(F_1) = x_3^3$,
$\pi_L(F_2) =x_2^3$, $\pi_L(F_3)=x_1^2x_2$ , 
$\pi_L(F_4)= 0$. The reduced Gr\"obner basis of $\pi_L(I)$ 
is $(\pi_L(F_1),\pi_L(F_2), \pi_L(F_3))$. We have to take 
out $\pi_L(F_4) = 0$.
\end{example}

\subsection{The non-Homogeneous Case}

The main feature of the homogeneous case is that 
if $\sigma$ is a term ordering of $x_i$-{\tt DegRev} type and~$F$ is a non zero 
homogeneous polynomial, then $x_i\, |\, F$ if and 
only if $x_i \,|\, \LT_\sigma(F)$. This fact implies that in 
Proposition~\ref{genhyper}
it suffices to remove $0$  from the set 
$\{\rho(G_1),\dots,\rho(G_s)\}$, 
and get the reduced Gr\"obner basis of $\pi_L(I)$
(see for instance Example~\ref{delete0}).
 Even if  $\sigma$ is a term ordering of $x_i$-{\tt DegRev} type,
in the non homogeneous case it may happen that 
$x_i$ divides $\LT_\sigma(f)$ but $x_i$ does not divide $f$.
For instance, if ${f=x^2-y}$ and $\sigma$ is any 
degre-compatible term ordering 
then~$x$ divides $x^2=\LT_\sigma(f)$, but it does not divide~$f$.  
These observations lead to a different approach of the 
hyperplane section problem in the non homogeneous case.
In particular, Assumptions~\ref{hyperplane} are modified as follows.

\begin{assumption}\label{stronghyperplane}
Let $P^{\mathstrut}=K[x_1,\dots,x_n]$, 
let $i \in \{1,\dots, n\}$, let $c_{i+1}, \dots, c_n, \gamma \in K$,
and let $L=x_i-\ell\ $ be the linear polynomial with
$\ell =  \sum_{j> i} c_jx_j+\gamma$ if $i<n$ , and  $\ell = \gamma$ if $i=n$.
We identify $P/(L)$ with the  ring ${\hat{P} = K[x_1, \dots, 
x_{i-1},x_{i+1},\dots,x_n]}$ via the isomorphism
defined by $\pi_L(x_i) = \ell$, $\pi_L(x_j) = x_j$
for $j\ne i$. 
\end{assumption}

\begin{theorem}\label{mainhyper}{\bf (Hyperplane Sections)}\\
Let~$\sigma$ be a term ordering on~$\mathbb T^n$ and,
under Assumptions~\ref{stronghyperplane}, let~$I$ be an ideal 
in the polynomial ring $P$,
let $G=\{g_1,\dots,g_s\}$ be a monic $\sigma$-Gr\"obner basis
of~$I$, and 
let~$\LT_\sigma(g_j) = \LT_{\hat{\sigma}}(\pi_L(g_j))$ for every $j=1,\dots, s$. 

\begin{enumerate}
\item[(a)] The linear polynomial $L$
does not divide zero modulo $I$.

\item[(b)] The set $\pi_L(G) = \{\pi_L(g_1),\dots,\pi_L(g_s) \}$ is 
a $\hat{\sigma}$-Gr\"obner basis of~${\pi_L(I)}_{\mathstrut}$.

\item[(c)] If $G$ is minimal, then also $\pi_L(G)$ is minimal.

\item[(d)] If $L=x_i-\gamma$ and $G$ is the reduced 
$\sigma$-Gr\"obner basis of $I$, then $\pi_L(G)$ 
is the reduced \hbox{$\hat{\sigma}$-Gr\"obner} basis of $\pi_L(I)$.

\end{enumerate}
\end{theorem}

\proof To prove (a) we assume, for contradiction, that $L$
is a zero divisor modulo $I$. Let $F$ be a non-zero monic polynomial with 
minimal leading term such that $FL\in I$. 
The assumption that 
$\LT_\sigma(g_j) = \LT_{\hat{\sigma}}(\pi_L(g_j))$ for every $j=1,\dots, s$
implies that $x_i$ does not divide any $\LT_\sigma(g_i)$.
Therefore from $\LT_\sigma(L)=x_i$ we deduce that there exist
$\ell \in \{1, \dots, s\}$ and a monic polynomial $H\in P$ 
with $\LT_\sigma(F) = \LT_\sigma(H)\LT_\sigma(g_\ell)$. Then $(F-Hg_\ell)L \in I$
and $\LT_\sigma(F-Hg_\ell) <_\sigma \LT_\sigma(F)$, a contradiction.

Next we prove (b). 
 It is clear that $\pi_L(G)$ 
generates $\pi_L(I)$, hence we need to prove that for every 
non-zero element~$f$ of $\pi_L(I)$ its leading term
$\LT_{\hat{\sigma}}(f)$ is divided by the leading term of 
an element of~$\pi_L(G)$.
For contradiction, suppose that there exists a non-zero monic
element $f \in \pi_L(I)$ such 
that~$\LT_{\hat{\sigma}}(f)$ is not divided by any leading 
term of the elements of $\pi_L(G)$, and let  $F \in I$ 
be a non-zero monic polynomial with 
minimal leading term such that $\pi_L(F)= f$.
A priori there are two possibilities: either $x_i\mid\LT_\sigma(F)$
or $x_i\nmid \LT_\sigma(F)$. If $x_i\mid \LT_\sigma(F)$
there exist an index $j\in \{1,\dots, s\}$ and a term $t \in \mathbb T_i$ 
such that we have ${\LT_\sigma(F) = x_i\cdot t\cdot \LT_\sigma(g_j)}$.
We let $H =F- L \cdot t\cdot  g_j$ and observe that  $\pi_L(H) = \pi_L(F)= f$,
since $\pi_L(L) = 0$.
On the other hand, $\LT_\sigma(H)<_\sigma\LT_\sigma(F)$ which 
contradicts the minimality of $F$. So this case is excluded and 
hence $x_i\nmid \LT_\sigma(F)$.
Since $\pi_L$ substitutes $x_i$ 
with a linear polynomial whose support contains only 
terms which are $\sigma$-smaller than~$x_i$, we deduce 
that  $\LT_\sigma(F) =  \LT_{\hat{\sigma}}(f)$. 
Since there exists $j$ such that $\LT_\sigma(g_j) \mid \LT_\sigma(F)$, 
we deduce that $\LT_{\hat{\sigma}}(\pi_L(g_j)) \mid \LT_{\hat{\sigma}}(f)$,
a contradiction. 

Claim (c) follows from (b) and the fact that the leading terms of the elements of both bases are the same.

Finally we prove (d). Let $t$ be a power product in the 
support of $\pi_L(g_j)-\LT_{\hat{\sigma}}(\pi_L(g_j))$. 
If $t=\pi_L(t)$ with $t$ in the support of $g_g-\LT_\sigma(g_j)$,
then $t$ is not a multiple of any $\LT_\sigma(g_j)$, hence of 
any $\LT_\sigma(\pi_L(g_j))$. If $t = \frac{1}{\gamma^a}\pi_L(x_i^at)$
with $x_i^at$ in the support of $g_g-\LT_\sigma(g_j)$, then  $x_i^at$ is not a multiple of any $\LT_\sigma(g_j)$, 
hence $t$ is not a multiple of any $\LT_\sigma(g_j)$, and so
$t$ is not a multiple of  any $\LT_\sigma(\pi_L(g_j))$ as well.
\endproof

In the theorem, the assumption that 
$\LT_\sigma(g_j) = \LT_{\hat{\sigma}}(\pi_L(g_j))$ for every $j=1,\dots, s$
is essential, as the following example 
shows.

\begin{example}\label{essentialassumption}
Let $P=K[x_1,x_2,x_3,x_4]$, let $f_1=x_2x_3-x_4$, $f_2 = x_1^3-2x_3^2$, 
let $I =(f_1, f_2)$, and let~$\sigma$
be any degree-compatible term ordering. Then $(f_1, f_2)$ is the 
reduced $\sigma$-Gr\"obner basis of~$I$. If we 
substitute $x_1$ with $x_3+x_4$, and let $f_2'$ be the polynomial 
obtained from $f_2$ with this substitution, then 
the reduced ${\tt DegRevLex}$-Gr\"obner basis
is  $(f_1,\  f_2',\ f_3)$ 
which differs from $(f_1, f_2')$
since it includes the \textbf{new}
polynomial $f_3=x_2x_4^3 +x_3^2x_4 +3x_3x_4^2 +3x_4^3 -2x_3x_4$.
\end{example}

In particular, the fact that 
$\LT_\sigma(g_j) = \LT_{\hat{\sigma}}(\pi_L(g_j))$ for $j=1,\dots, s$ 
if and only if~$x_i$ does not divide any leading term of the 
elements of $G$, is essential in the proof that minimality of~$G$ implies 
minimality of $\pi_L(G)$. 
However,  for a general $L$ the conclusion of  statement (d) of the theorem 
does  not hold, as the following example shows. 

\begin{example}\label{non reduced}
The set  $G =\{ x_2^3-x_1^2, \ x_3^2-1\}$ is the reduced 
$\sigma$-Gr\"obner 
basis for every degree-compatible term ordering~$\sigma$ 
with $x_1>_\sigma x_2>_\sigma x_3$, but if $L=x_1-x_3$ then 
$\pi_L(G) = \{{x_2^3}^{\mathstrut}-x_3^2, \ x_3^2-1\}$ is not reduced.
\end{example}

\subsection{Lifting}

We are going to prove a sort
of converse of Theorem~\ref{mainhyper}.

\begin{theorem}\label{lifting}{\bf (Lifting Gr\"obner Bases)}\\
Let~$\sigma$ be a term ordering on~$\mathbb T^n$ and,
under Assumptions~\ref{stronghyperplane},  let $I$ be an ideal in $P$
such that~$L$ does not divide zero modulo $I$, 
let $G=\{g_1,\dots, g_s\} \subset I$ be such that
$\pi_L(G) = \{\pi_L(g_1), \dots, \pi_L(g_s)\}$ is
a $\hat{\sigma}$-Gr\"obner
basis of~$\pi_L(I)$, and 
$\LT_\sigma(g_j) = \LT_{\hat{\sigma}}(\pi_L(g_j))$
 for $j=1,\dots, s$.
 
 \begin{enumerate}
\item[(a)] 
The set $G$ is a $\sigma$-Gr\"obner basis of $I$.

\item[(b)] If $\pi_L(G)$ is minimal, then also $G$ is minimal.

\item[(c)] If $L = x_i-\gamma$ and $\pi_L(G)$ is the reduced
$\hat{\sigma}$-Gr\"obner basis of $\pi_L(I)$, then 
$G$ is the reduced  $\sigma$-Gr\"obner basis of $I$.

\end{enumerate}
\end{theorem}

\proof We prove (a) by contradiction. Suppose that there 
exists a monic non-zero polynomial $F \in I$ with minimal leading 
term among the 
polynomials in $I$ such that $\LT_\sigma(F)$ is not divisible 
by any leading term of the elements of $G$, and
let $f = \pi_L(F)$. If $f = 0$ then there exists $H \in P$ with 
$F = HL$, and the assumption about $L$ implies 
that $H \in I$. Moreover $\LT_\sigma(H)\mid \LT_\sigma(F)$, 
a contradiction. 
Therefore $f=0$ is excluded, hence there exist suitable
polynomials $h_j\in \hat{P}$ such that 
$f$ can be written as $f= \sum_{j=1}^sh_j\pi_L(g_j)$
with $\LT_{\hat{\sigma}}(f)= 
\max_{j=1}^s\{\LT_{\hat{\sigma}}(h_j\pi_L(g_j))\}$. 
If we let  $U = F-\sum_{j=1}^sh_jg_j$, we get $\pi_L(U) = 
f-\sum_{j=1}^sh_j\pi_L(g_j) = 0$. Consequently
there exists $H \in P$ with 
$U = HL$ and the assumption about $L$ implies 
that $H \in I$. We examine the two possible cases.

{\it Case 1}: $H = 0$. In this case $F = \sum_{j=1}^sh_jg_j$. 
We know that  $\LT_{\hat{\sigma}}(f)
= \max_{j=1}^s\{\LT_{\hat{\sigma}}(h_j\pi_L(g_j)\}$, hence
there exists at least one index $\ell$ such that 
$\LT_{\hat{\sigma}}(f) = \LT_{\hat{\sigma}}(h_\ell\pi_L(g_\ell))$. 
On the other hand we have $\LT_{\hat{\sigma}}(h_\ell \pi_L(g_\ell)) 
= \LT_\sigma(h_\ell g_\ell)$, hence $\LT_\sigma(F) 
=  \LT_\sigma(h_\ell g_\ell)$, so that $\LT_\sigma(g_\ell)\mid \LT_\sigma(F)$,
a contradiction.

{\it Case 2}: $H \ne 0$.  Since $\pi_L$ 
can only lower the leading term of a polynomial, 
we have the equality $ \LT_\sigma(F)=\LT_\sigma(U)$, hence 
$\LT_\sigma(F)= \LT_\sigma(HL)$. But then 
$\LT_\sigma(H) \mid \LT_\sigma(F)$, and $H \in I$, 
a contradiction.

Claim (b) follows from (a) and the fact that the leading terms of the elements of both bases are the same.

To prove (c) we let $t$ be  a power product in the 
support of $g_j-\LT_\sigma(g_j)$. We have $t = x_i^at'$ 
with $x_i\nmid t'$.
Then $\pi_L(t) = c_o^at'$. We know that $t'$ 
is not divided by any leading term
of the $\pi_L(g_j)$, hence also $t$ is not divided by 
any leading term of the $g_j$.
\endproof

The following examples show the tightness of the assumptions 
in the above theorem.

\begin{example}\label{nonzerodivnecess}
Let $P=K[x_1,x_2,x_3,x_4]$ and $L= x_2-x_4$.
Then let $\sigma$ be any degree-compatible term ordering,
let $G= \{x_1^2,\ x_1x_3-x_2, \ x_1x_4, \ x_4^2\}$, let $I$ be 
the ideal generated by~$G$, and let
${\pi_L(G)= \{x_1^2 ,\ x_1x_3-x_4, \ x_1x_4, \ x_4^2  \}}$.
We have  
$${x_1^2x_3-x_1(x_1x_3-x_2)-x_1x_4 = x_1(x_2-x_4) = x_1L \in I}$$
which implies that $L$ divides zero modulo $I$, so that all the 
hypotheses are satisfied except one.
And we see that $\pi_L(G)$ is the reduced $\hat{\sigma}$-Gr\"obner 
basis of $I_L$, while $G$ is not a $\sigma$-Gr\"obner basis of $I$, since
the reduced \hbox{$\sigma$-Gr\"obner} basis of $I$ 
is $ \{x_1^2,\ x_1x_3-x_2, \ x_1x_4, \ x_4^2, \  \mathbf{x_1x_2,\  x_2^2}\}$.
\end{example}

\begin{example}
Let $P=K[x_1,x_2,x_3,x_4]$ and $L= x_2-x_4$,
let $\sigma$ be any degree-compatible term ordering on $\mathbb T^n$,
let $G= \{x_2^2-x_3^2, \ x_1x_2\}$, let $I$ be 
the ideal generated by~$G$, and finally  let
$\pi_L(G)= \{-x_3^2+x_4^2,\  x_1x_4 \}$.
We observe that all the hypotheses are satisfied, 
except the fact that $\LT_\sigma(g_i) = \LT_{\hat{\sigma}}(\hat{g}_i)$
 for $i =1,\dots, s$.
And we see that 
$\pi_L(G)$ is the reduced $\hat{\sigma}$-Gr\"obner basis of $I_L$,
while~$G$ is not a $\sigma$-Gr\"obner basis of $I$, since
the reduced $\sigma$-Gr\"obner basis of $I$ is
$ \{x_2^2-x_3^2, \ x_1x_2, \ \mathbf{x_2x_3^2,\  x_3^4} \}$.
\end{example}

\begin{example}\label{notreduced}
Let $P=K[x_1,x_2,x_3,x_4]$ and let $\ell = x_2$, $L=x_1-\ell = x_1-x_2$.
Let $\sigma$ be any degree-compatible term ordering, and  let 
$G=\{x_2^3+x_1x_3-x_2x_3, \ x_3\}$. Then $\pi_L(G) = \{x_2^3, x_3\}$ is the reduced $\hat{\sigma}$-Gr\"obner basis, while $G$ is not reduced.
\end{example}

\subsection{Common Lifting}
In the following we consider  the lifting of Gr\"obner bases as 
described in Theorem~\ref{lifting} and
start  investigating how to make it explicit.

\begin{assumption}\label{multilift}
Let  $P=K[x_1,\dots,x_n]$,  $i \in \{1,\dots, n\}$, 
$c_{i+1}, \dots, c_n \in K$, $N \in \N$. Moreover
let $\gamma_1, \dots, \gamma_N$ be distinct elements of $K$, 
and  for $k =1,\dots, N$ let $L_k = x_i-\ell_k$ be linear polynomials
with $\ell_k = \sum_{j> i} c_jx_j+\gamma_k$ if $i<n$ , 
and  $\ell = \gamma_k$ if $i=n$.
We identify $P/(L_k)$ with the polynomial ring $\hat{P} = K[x_1, \dots, 
x_{i-1},x_{i+1},\dots,n]$ via the isomorphism induced by 
$\pi_{L_k}(x_i) = \ell_k$ and $\pi_{L_k}(x_j) = x_j$ for~$j\ne i$.
\end{assumption}

\begin{definition}\label{common}
Under Assumptions~\ref{multilift},
let  $\hat{g}_1, \dots, \hat{g}_N \in \hat{P}$ and let $g \in P$ be 
a polynomial such that $\pi_{L_k}(g) = \hat{g}_k$ for $k = 1,\dots, N$.
Then $g$ is called  a \textbf{common lifting} of $\hat{g}_1, \dots, \hat{g}_N$.
\end{definition}

\begin{remark}\label{commonlowdeg}
The linear polynomials are pairwise coprime. 
Therefore the  Chinese Remainder 
Theorem (see~\cite[Lemma 3.7.4]{KR2})
implies that all the common liftings of the $g_k$ differ by a multiple 
of $\prod_{k=1}^N L_k$. Consequently there exists at most one common lifting of degree less than $N$. However, even if the degrees of the 
$\hat{g}_i$ are less than $N$, a common lifting of degree less than $N$ may not exist, as the following example shows.
\end{remark}

\begin{example}
Let $P = K[x,y]$, $L_0 = y$, $L_1=y-1$, $L_2 = y-2$.
Then let $\hat{g}_1 = x$,  $\hat{g}_2 = x+1$,  $\hat{g}_3= x+4$, and
observe that their degree is less than $N=2$.
We compute a common lifting  and get $g = y^2+x$,  as the only 
one of degree less than 3. Therefore there is no common lifting of degree 
less than $2$.
\end{example}

\begin{theorem}\label{commonliftingGroebner}
{\bf (Common Lifting of Gr\"obner Bases)}\\
Under Assumptions~\ref{multilift}, let~$\sigma$ be a term 
ordering on~$\mathbb T^n$ and let $I$ be an ideal in $P$.
\begin{enumerate}
\item[(a)] If $\gamma_1, \dots, \gamma_N$  are sufficiently generic, 
all the minimal $\hat{\sigma}$-Gr\"obner bases of the ideals $\pi_{L_k}(I)$ 
share the same number of elements and the same leading terms, 
say $t_1, \dots, t_s$.

\item[(b)] If $N \gg 0$,  at least one of the $L_k$ does 
not divide zero modulo $I$.

\item[(c)] Let $g_1, \dots, g_s$ be  common 
liftings of the corresponding elements in 
minimal $\hat{\sigma}$-Gr\"obner
bases of the ideals $\pi_{L_k}(I)$. 
If~$g_i\!\in\! I$ and $\, \LT_\sigma(g_i) = t_i\,$ 
for $i=1,\dots, s$ then $\{g_1, \dots, g_s\}$
is a minimal $\sigma$-Gr\"obner basis of $I$.

\item[(d)] Let $L=x_i-\gamma$ and let $g_1, \dots, g_s$ be  common 
liftings of the corresponding elements in the
reduced $\hat{\sigma}$-Gr\"obner
bases of the ideals $\pi_{L_k}(I)$. 
If~$g_i\!\in\! I$ and $\, \LT_\sigma(g_i) = t_i\,$ 
for $i=1,\dots, s$ then $\{g_1, \dots, g_s\}$
is the reduced $\sigma$-Gr\"obner basis of $I$.
\end{enumerate}
\end{theorem}

\proof To prove claim (a), we let $a$ be a free parameter, 
let $L_a=x_i-( \sum_{j> i} c_jx_j +a)$, 
and let~$I_a$ be the ideal $I + (L_a)$ in the polynomial ring $K(a)[x_1, \dots, x_n]$. 
The $\sigma$-reduced Gr\"obner basis  of $I_a$ consists of $L_a$ 
and  polynomials in $K(a)[x_1, \dots,x_{i-1},x_{i+1},\dots , x_n]$. 
It evaluates to the reduced Gr\"obner basis 
of the corresponding ideal for almost all values of $a$ which implies 
that  the ideals~$\pi_{L_k}(I)$ share the same  leading term ideals 
(see~\cite[Proposition 2.3]{BR} for a more general argument). 

Claim (b) follows from the fact that each primary component of $I$ can 
contain at most one of the linear polynomials $L_k$, 
since any pair of them generate the unit ideal.

Claim (c) follows from (b) and Theorem~\ref{lifting}.b. 

Claim (d) follows from (b) and Theorem~\ref{lifting}.c. 
\endproof

Here we show an interesting example.

\begin{example}\label{zitrus}{\bf (Zitrus)}\\
There is a well-known example of a surface which represents a lemon
(see for instance the web page\ 
{\tt http://imaginary.org/gallery/herwig-hauser-classic}).
Its equation is the following ${F:=x^2+z^2 - y^3(1-y)^3=0}$. We cut it with 
a sufficiently high number of \textit{parallel hyperplanes}. In our case we choose
$z-\gamma =0$ for $\gamma \in \{-5,-4,-3,-2, 2,3,4,5\}$ and get the 
hyperplane sections defined by the following eight ideals:
$(y+5, \ x^2 +z^2 +27000)$, 
$(y+4, \ x^2 +z^2 +8000)$, 
$(y+3, \ x^2 +z^2 +1728)$, 
$(y+2, \ x^2 +z^2 +216)$, 
$(y-2, \ x^2 +z^2 +8)$, 
$(y-3, \ x^2 +z^2 +216)$, 
$(y-4, \ x^2 +z^2 +1728)$, 
$(y-5, \ x^2 +z^2 +8000)$. We use a \cocoa\ (see~\cite{Co}) script to compute the reconstruction
according to Theorem~\ref{commonliftingGroebner}, and indeed we get $F$ back. As a matter of curiosity, 
we observe that \textit{the real lemon} is reconstructed using eight slices with \textit{no real points}.
\end{example}

\begin{remark}
In general, if we want to use Theorem~\ref{commonliftingGroebner} to compute 
a Gr\"obner basis of~$I$,
we need to verify that the polynomials $g_k$ have the correct leading term,
and this is easy to do.  We also need to verify that they are in~$I$, and
in general this is a limit to the usefulness of the theorem.
Nevertheless there are nice situations where this verification 
can be done easily.
For instance, if we have a Gr\"obner basis $G$ of~$I$ and want to
compute the reduced Gr\"obner basis of~$I$ with respect to another
term ordering, then checking that the $g_k$ belong to~$I$ entails the
simple verification that their normal forms with respect to~$G$ are zero.
A favourable situation is the following.
If the ideal~$I$ is known via a parametrization,
then checking that the $g_k$ belong to $I$ requires only evaluating them at the parametric
expressions of the coordinates.
Let us show an example where the expected output
is a single polynomial.  
\end{remark}

\begin{example}\label{sectparams}{\bf (Rational Surface)}\\
Let $S$ be the affine surface in $\mathbb{A}^3_\Q$ given parametrically by 
{\small
$$
\begin{cases}
x\ =&s^5-st^3-t \\
y\ =&st^2-s \\
z\ =&s^4-t^2
\end{cases}
$$
}
The implicit equation of $S$ can be computed using a standard elimination procedure. 
We do it in~\cocoa\ and get the implicit equation $F=0$ where $F$ 
is the polynomial displayed below. It has degree~$14$ and its  support contains 319 power products.
Using a procedure suggested by the theorem and the remark, 
we can slice the surface with several hyperplanes 
parallel to~${z=0}$, compute the cartesian equation of the corresponding 
curve viewed as a curve in the affine plane, and then reconstruct the 
equation of the surface.
This procedure computes the polynomial $F$ using 
approximately $\frac{1}{120}$ of the time 
used by the standard procedure, and deserves further investigation.

\smallskip
{ \scriptsize
$F=
y^{14} +10y^{13} -12x^2y^8z^3 +34xy^9z^3 -10y^{10}z^3 +8xy^5z^7 -11y^6z^7 -y^2z^{11} +45y^{12} +8x^5y^5z^2 -65x^4y^6z^2
 +132x^3y^7z^2 -32x^2y^8z^2 -34xy^9z^2 -3y^10z^2 -72x^2y^7z^3 +204xy^8z^3 -44y^9z^3 +2x^4y^2z^6 -2x^3y^3z^6 +45x^2y^4z^6 -78xy^5z^6 +57y^6z^6 +32xy^4z^7 -44y^5z^7 -x^2z^10 -8y^2z^10 -2yz^11 +120y^11 -x^8y^2z +18x^7y^3z -86x^6y^4z +120x^5y^5z +38x^4y^6z -92x^3y^7z -28x^2y^8z +4xy^9z +32x^5y^4z^2 -260x^4y^5z^2 +464x^3y^6z^2 +112x^2y^7z^2 -236xy^8z^2 -48y^9z^2 -180x^2y^6z^3 +510xy^7z^3 -50y^8z^3 +2x^6z^5 -10x^5yz^5 +62x^4y^2z^5 -44x^3y^3z^5 +31x^2y^4z^5 +54xy^5z^5 +4y^6z^5 +4x^4yz^6 -4x^3y^2z^6 +138x^2y^3z^6 -196xy^4z^6 +164y^5z^6 +48xy^3z^7 -66y^4z^7 -10x^2z^9 -24y^2z^9 -18yz^10 -z^11 -x^10 +10x^9y -31x^8y^2 +18x^7y^3 +49x^6y^4 -26x^5y^5 -42x^4y^6 -4x^3y^7 +7x^2y^8 +2xy^9 +210y^10 -2x^8yz +36x^7y^2z -156x^6y^3z +56x^5y^4z +544x^4y^5z -296x^3y^6z -312x^2y^7z -24xy^8z -4y^9z +48x^5y^3z^2 -390x^4y^4z^2 +520x^3y^5z^2 +810x^2y^6z^2 -490xy^7z^2 -233y^8z^2 -240x^2y^5z^3 +680xy^6z^3 +64y^7z^3 +10x^6z^4 -2x^5yz^4 +127x^4y^2z^4 -78x^3y^3z^4 +163x^2y^4z^4 -12xy^5z^4 +5y^6z^4 +64x^4yz^5 +24x^3y^2z^5 +164x^2y^3z^5 +88xy^4z^5 +116y^5z^5 +2x^4z^6 -2x^3yz^6 +153x^2y^2z^6 -140xy^3z^6 +183y^4z^6 +32xy^2z^7 -44y^3z^7 -41x^2z^8 -32y^2z^8 -64yz^9 -10z^10 +18x^8y -124x^7y^2 +204x^6y^3 +104x^5y^4 -236x^4y^5 -120x^3y^6 +24x^2y^7 +16xy^8 +252y^9 -x^8z +18x^7yz -50x^6y^2z -286x^5y^3z +950x^4y^4z +216x^3y^5z -908x^2y^6z -280xy^7z -44y^8z +32x^5y^2z^2 -260x^4y^3z^2 +80x^3y^4z^2 +1480x^2y^5z^2 -232xy^6z^2 -518y^7z^2 +32x^6z^3 +40x^5yz^3 +128x^4y^2z^3 +66x^3y^3z^3 -134x^2y^4z^3 +518xy^5z^3 +231y^6z^3 +32x^5z^4 +174x^4yz^4 +164x^3y^2z^4 +364x^2y^3z^4 +164xy^4z^4 +46y^5z^4 +22x^4z^5 +80x^3yz^5 +327x^2y^2z^5 +178xy^3z^5 +267y^4z^5 +72x^2yz^6 -4xy^2z^6 +126y^3z^6 -88x^2z^7 +8xyz^7 -27y^2z^7 -112yz^8 -41z^9 +4x^8 -22x^7y -95x^6y^2 +496x^5y^3 -185x^4y^4 -484x^3y^5 -58x^2y^6 +40xy^7 +209y^8 +24x^6yz -260x^5y^2z +380x^4y^3z +1120x^3y^4z -816x^2y^5z -776xy^6z -180y^7z +56x^6z^2 +88x^5yz^2 +59x^4y^2z^2 -150x^3y^3z^2 +1195x^2y^4z^2 +450xy^5z^2 -598y^6z^2 +128x^5z^3 +296x^4yz^3 +424x^3y^2z^3 +208x^2y^3z^3 +308xy^4z^3 +280y^5z^3 +92x^4z^4 +362x^3yz^4 +641x^2y^2z^4 +514xy^3z^4 +202y^4z^4 +32x^3z^5 +246x^2yz^5 +308xy^2z^5 +300y^3z^5 -92x^2z^6 +18xyz^6 +76y^2z^6 -96yz^7 -88z^8 -46x^6y +128x^5y^2 +340x^4y^3 -560x^3y^4 -316x^2y^5 +16xy^6 +112y^7 +52x^6z +42x^5yz -38x^4y^2z +926x^3y^3z +97x^2y^4z -884xy^5z -362y^6z +224x^5z^2 +368x^4yz^2 +248x^3y^2z^2 +546x^2y^3z^2 +692xy^4z^2 -358y^5z^2 +248x^4z^3 +680x^3yz^3 +736x^2y^2z^3 +456xy^3z^3 +257y^4z^3 +160x^3z^4 +602x^2yz^4 +776xy^2z^4 +368y^3z^4 -2x^2z^5 +194xyz^5 +233y^2z^5 -4yz^6 -104z^7 +10x^6 +46x^5y +179x^4y^2 -82x^3y^3 -369x^2y^4 -64xy^5 +21y^6 +192x^5z +220x^4yz +388x^3y^2z +464x^2y^3z -432xy^4z -396y^5z +384x^4z^2 +624x^3yz^2 +502x^2y^2z^2 +486xy^3z^2 -78y^4z^2 +352x^3z^3 +872x^2yz^3 +792xy^2z^3 +322y^3z^3 +186x^2z^4 +536xyz^4 +384y^2z^4 +32xz^5 +98yz^5 -62z^6 +64x^5 +134x^4y +28x^3y^2 -112x^2y^3 -64xy^4 -22y^5 +298x^4z +358x^3yz +302x^2y^2z -82xy^3z -257y^4z +416x^3z^2 +640x^2yz^2 +448xy^2z^2 +82y^3z^2 +368x^2z^3 +688xyz^3 +376y^2z^3 +128xz^4 +200yz^4 -6z^5 +100x^4 +94x^3y -21x^2y^2 -15y^4 +256x^3z +248x^2yz +28xy^2z -112y^3z +344x^2z^2 +440xyz^2 +167y^2z^2 +224xz^3 +240yz^3 +32z^4 +64x^3 +22x^2y +148x^2z +94xyz -21y^2z +192xz^2 +144yz^2 +56z^3 +15x^2 +64xz +22yz +48z^2 +15z
$
}
\end{example}

\bigskip\goodbreak

\section{Families of Schemes and the Hough Transform}
\label{Families of Schemes}

In this section we consider families of algebraic schemes and recall
the necessary tools to introduce the notion of Hough transform.

\subsection{Families of Schemes}
As said in the introduction, the notation is borrowed
from~\cite{KR1} and~\cite{KR2}.
We start the section by recalling some definitions taken 
from~\cite{BR}.
We let~$x_1, \dots, x_n$ be
indeterminates, and most of the time in the following
we use the notation $\x = (x_1, \dots, x_n)$.
If~$K$ is a field, the multivariate
polynomial ring~$P=K[x_1,\dots,x_n]$ is
denoted by~$K[\x]$. If $f_1(\x), \dots, f_k(\x)$
are polynomials in~$P$,
the set~$\{f_1(\x), \dots, f_k(\x)\}$ is denoted by~$\f(\x)$, and
$\f(\x)=0$ is called a system of polynomial equations.

Moreover, we let $\aaa = (a_1, \dots, a_m)$
be an $m$-tuple of indeterminates which will play the role of
parameters. If we are given polynomials 
$F_1(\aaa, \x), \ldots, F_k(\aaa, \x)$ in~$K[\aaa, \x]$, 
we let~${F(\aaa, \x)=0}$ be the corresponding 
{\bf set of systems of polynomial  equations} 
parametrized by~$\aaa$, and 
the ideal generated by~$F(\aaa,\x)$ in $K[\aaa,\x]$
is denoted by $I(\aaa,\x)$. 

Let $\Sc$ be the affine scheme of the $\aaa$-parameters,
$R$ its coordinate ring, and
$\mathcal{F}$ the affine scheme $\Spec(K[\aaa,\x]/I(\aaa,\x))$.
Then there exists a morphism of schemes ${\Phi: \mathcal{F}  \To \Sc}$,
or equivalently a $K$-algebra homomorphism
$\phi: R \To K[\aaa,\x]/I(\aaa,\x)$. The morphism $\Phi$, 
and $\mathcal{F}$ itself
if the context is clear,  is called a \textbf{family 
of sub-schemes of $\mathbb{A}^m$}. 

\begin{definition}\label{independparams}
If $\Sc= \mathbb A^m$ and 
${I(\aaa,\x)\cap K[\aaa] = (0)}$  the parameters~$\aaa$ are 
said to be  {\bf independent with respect to $\mathcal{F}$}, or simply 
{\bf independent}.
\end{definition}

\begin{remark}\label{dominant}
According to the above definition, the parameters $\aaa$ are independent
if and only if the $K$-algebra homomorphism 
$\phi:K[\aaa] \To  K[\aaa,\x]/I(\aaa,\x)$ is injective. 
Therefore  the parameters~$\aaa$ are independent
if and only if the morphism 
$\Phi: \mathcal{F}\To \mathbb A^m$ is dominant.
\end{remark}

\begin{definition}\label{iflat}
Let $\f(\x)$ be a set of polynomials in $P$, and
let $F(\aaa, \x)$ define a family 
which specializes to $\f(\x)$ for a suitable choice of the parameters.
Then  let $I=(\f(\x))$, let $I(\aaa,\x) = (F(\aaa, \x))$, 
let $\mathcal{F} = \Spec(K[\aaa,\x]/I(\aaa,\x))$,
let~$\mathcal{S}=\mathbb A^m$, assume that 
the $\aaa$-parameters are independent,
and let 
$\Phi: \mathcal{F}  \To \mathbb{A}^m$ 
be the associated morphism of schemes.
A dense Zariski-open subscheme~$\U$ of~$\mathbb{A}^m$
with the property that~${\Phi^{-1}(\U) \To \U}$ is
free is said to be an~{\bf $\mathcal F$-free} 
subscheme of~$\mathbb{A}^m$ or
simply an~$\mathcal F$-free  scheme.
\end{definition}

\begin{assumption}\label{domfamily}
Let $\Phi: \mathcal{F} \To \mathbb A^m$ be a 
 family of sub-schemes of $\mathbb A^n$ parametrized 
 by~$\mathbb{A}^m$. Then  let $\sigma$  be
a term ordering on $\mathbb T^n$,
let $G_\sigma(\aaa, \x)$ be the reduced $\sigma$-Gr\"obner basis
of the extended ideal $I(\aaa, \x)K(\aaa)[\x]$, and let $d_\sigma(\aaa)$ be the 
least common multiple of all the denominators of the coefficients 
of the polynomials in $G(\aaa,\x)$.
\end{assumption}

\begin{proposition}\label{flatness}
Under Assumptions~\ref{domfamily} the open 
subscheme $\U_\sigma=\mathbb{A}^m\setminus \{d_\sigma(\aaa) = 0\}$  
of $\mathbb{A}^m$ is~$\mathcal{F}$-free.
\end{proposition}

\begin{proof}
See~\cite[Proposition 2.3]{BR}.
\end{proof}

\begin{definition}\label{restriction}
The set $G_\sigma(\aaa,\x)$ is called the {\bf universal 
reduced $\sigma$-Gr\"obner basis of $\mathcal{F}$}.
We say that $d_\sigma(\aaa)$ is 
the \hbox{\bf $\sigma$-denominator} of $\Phi$,
that $\Phi_{|d_\sigma(\aaa)}$ is the {\bf $\sigma$-free 
restriction} of $\Phi$, and that
$\U_\sigma$ is  the 
{\bf $\sigma$-free set} of the family $\mathcal{F}$. 
\end{definition}

\bigskip\goodbreak

\begin{proposition}\label{dependent}
Under Assumptions~\ref{domfamily}
the following conditions are equivalent.
\begin{enumerate}
\item[(a)] The $\aaa$-parameters are dependent 
with respect to $\mathcal{F}$.
\item[(b)] We have $I(\aaa, \x)K(\aaa)[\x] = (1)$.
\item[(c)] We have $G_\sigma(\aaa,\x) = \{1\}$.
\end{enumerate} 
\end{proposition}

\proof
The equivalence between (b) and (c) is a standard (easy) 
fact in computer algebra, so let us prove the equivalence 
between (a) and (b).
If the $\aaa$-parameters are dependent 
with respect to~$\mathcal{F}$ then $I(\aaa, \x)\cap K[\aaa]$ 
contains a non-zero polynomial, say $f(\aaa)$. Then $I(\aaa, \x)K(\aaa)[\x] $ 
contains a non-zero constant, hence it is the unit ideal. 
Conversely, if $I(\aaa, \x)K(\aaa)[\x] = (1)$ then we may write $1$ 
as a combinations of polynomials in $I(\aaa, \x)$ with coefficients in $K(\aaa)[x]$. 
Hence there exists a common denominator, say $f(\aaa)$ 
such that $f(\aaa) = f(\aaa) \cdot 1 \in I(\aaa, \x)$, and the proof is complete.
\endproof

Let  $\Phi: \mathcal{F} \To \mathbb A^m$ be a dominant family 
of sub-schemes of $\mathbf{A}^m$.
It corresponds to a $K$-algebra 
homomorphism  ${\phi: K[\aaa] \To K[\aaa,\x]/I(\aaa,\x)}$.
As observed in Remark~\ref{dominant}, the dominance 
implies that the $\aaa$-parameters are independent, 
therefore  $\phi$ is injective.
If we fix $\bfalpha = (\alpha_1, \dots, \alpha_m)$, i.e.\ a 
rational ``parameter point'' in $\mathbb A^m$,
we get $\Spec(K[\bfalpha,\x]/I(\bfalpha,\x))$, a special fiber of $\Phi$, 
hence a special member 
of the family.  Clearly we have the equality $K[\bfalpha,\x]=K[\x]$ so that 
$I(\bfalpha,\x)$ can be seen as an ideal in $K[\x]$. With this convention
we denote the scheme $\Spec(K[\x]/I(\bfalpha,\x))$ by~$\X_{\bfalpha,\x}$.

On the other hand,  there exists another 
morphism $\Psi: \mathcal{F} \To \mathbb A^n$  which 
corresponds to the \hbox{$K$-algebra} 
homomorphism $\psi: K[\x] \To K[\aaa,\x]/I(\aaa,\x)$.
If we fix a rational point  $p=(\xi_1, \dots, \xi_n)$  
in $\mathbb A^n$, we 
get a special fiber of the morphism~$\Psi$, 
namely $\Spec(K[\aaa,p]/I(\aaa,p))$.  
Clearly we have $K[\aaa,p]=K[\aaa]$ 
so that ~$I(\aaa,p)$ can be seen as an ideal in $K[\aaa]$. 
With this convention we denote the 
scheme $\Spec(K[\aaa]/I(\aaa,p))$ by~$\Gamma_{\aaa,p}$.

\begin{definition}
Let $G=G_\sigma(\aaa,\x)$  be the universal 
reduced $\sigma$-Gr\"obner basis of $\mathcal{F}$, 
listed with $\sigma$-increasing leading terms. 
The corresponding list of non-constant 
coefficients of $G$  is denoted by  ${\rm NCC}_G$
and called the {\bf non constant coefficient list} of $G$.
Moreover, if~$\alpha \in \mathcal{U}_\sigma$ then ${\rm NCC}_G(\alpha)$ 
is the list obtained by $\alpha$-evaluating the elements  
${\rm NCC}_G$.
\end{definition}

\begin{example}\label{excoefflist}
Let $\mathcal F$ be the family of subschemes of $\mathbb A^2$ which is defined by the following ideal 
$I(\aaa, \x) = (x_1^2+a_1^2x_2-a_2,\ 
x_2^3+(a_3^2+1)x_1^2+x_1+a_1a_3x_2-1)$, and let $\sigma$ be a degree-compatible term ordering with $x_1>_\sigma x_2>_\sigma x_3$. 
Then ${\rm NCC}_G = [a_1^2, -a_2, a_3^2+1,a_1a_3]$.
\end{example}

The main property of the non constant coefficient 
list of $G_\sigma(\aaa,\x)$ is described as follows.

\begin{proposition}\label{coefflistuniqe}
Under Assumptions~\ref{domfamily}, the
correspondence between  
$\{\X_{\alpha,\x}\, \mid\,  \alpha \in \mathcal{U}_\sigma\}$ 
and~${\rm NCC}_G$ which is defined by 
sending $\X_{\alpha,\x}$ to ${\rm NCC}_G(\alpha)$ is bijective.
\end{proposition}

\proof First we show that the universal 
reduced $\sigma$-Gr\"obner basis of $\, \mathcal{F}$ 
specializes to the reduced $\sigma$-Gr\"obner 
basis of each fiber~$\X_{\alpha,\x}$. The reason is that when we 
specialize we do not affect the leading terms and we do not 
add new elements to the support of the polynomials involved.
Then the conclusion follows from the fact that the 
reduced $\sigma$-Gr\"obner  basis  of an ideal is 
unique (see~\cite[Theorem 2.4.13]{KR1}).
\endproof

Proposition~\ref{coefflistuniqe} suggests the following definition.

\begin{definition}\label{schemefamily}
Let $\mathcal{U}_\sigma$ be the $\sigma$-free set of $\mathcal{F}$,
let $G=G_\sigma(\aaa,\x)$ be the universal 
reduced \hbox{$\sigma$-Gr\"obner} of $\mathcal{F}$, and 
let ${\rm NCC}_G$  be the non
constant coefficient list of $G$.  Then the scheme 
parametrized by ${\rm NCC}_G$ is called the {\bf $\sigma$-scheme 
of $\mathcal F$}. 
If ${\rm NCC}_G=(\frac{f_1(\aaa)}{d_1(\aaa)}, \dots, 
\frac{f_s(\aaa)}{d_s(\aaa)})$, then the $\sigma$-scheme
of $\mathcal{F}$ is represented parametrically by
$$
y_1=\frac{f_1(\aaa)}{d_1(\aaa)}, \dots, \ y_s
=\frac{f_s(\aaa)}{d_s(\aaa)}
$$
which is called the {\bf parametric representation
of the $\sigma$-scheme of $\mathcal{F}$}.
\end{definition}

\begin{example}\label{sigma scheme}
Let $\mathcal{F} = \Spec(K[\aaa,\x]/I(\aaa,\x)$ 
where $I(\aaa,\x) = (x^2+a_1^2x+a_1a_2y+a_2^2)$, 
and let~$\sigma$ be a degree-compatible term ordering. 
Then ${\rm NCC}_G = ({a_1^2}_{\mathstrut}, a_1a_2, a_2^2)$
and the $\sigma$-scheme of $\mathcal{F}$ is the affine cone $\X_\sigma$
represented by $y_1=a_1^2$,\ $y_2 = a_1a_2$,\  $y_3 = a_2^2$.
Its defining ideal in $K[y_1,y_2,y_3]$ is generated by $y_2^2-y_1y_3$,
and we have $\dim(\X_\sigma) = 2$.
\end{example}

Using Proposition~\ref{coefflistuniqe} and the theory of Gr\"obner fans
(see~\cite{MR}), we get the following result.

\goodbreak

\begin{corollary}\label{H-fan}
Let $\Phi: \mathcal{F} \To \mathbb A^m$ 
be a dominant family of  sub-schemes of $\mathbb A^n$.
\begin{enumerate}
\item[(a)] For every  term ordering $\sigma$,  
the $\sigma$-scheme of $\,\mathcal{F}$
represents the generic fibers  of $\,\mathcal{F}$.
\item[(b)] The set of $\sigma$-schemes of $\,\mathcal{F}$ is finite.
\end{enumerate}
\end{corollary}

\proof Claim (a) is a restatement of the proposition. Claim (b) 
follows from the  theory of Gr\"obner fans which entails there 
is only a finite number of reduced Gr\"obner bases
of the ideal~$I(\aaa, \x)$.
\endproof

\begin{remark}\label{notbijective}
The statement of the proposition
does not imply that there is a bijection between $\mathcal U_\sigma$ 
and ${\rm NCC}_G$, as Example~\ref{excoefflist} shows . 
For instance in that example, for $(a_1,a_2,a_3)=(1,1,1)$ and $(a_1,a_2,a_3)=(-1,1,-1)$ we get the same fiber.
The reason is that  the proposition treats
$\{\X_{\alpha,\x}\, \mid\,  \alpha \in \mathcal{U}_\sigma\}$ 
as a set. It means that if we have 
$\X_{\alpha,\x}=\X_{\alpha'\!,\x}$ we view the two fibers as a single 
element of the set.
\end{remark}

\medskip

\subsection{Hyperplane Sections and Families}
\label{HyperSecFam}

The setting of this subsection  is the following. 

\begin{assumption}\label{familyOfSch}
Let $\,\mathcal{F}$
be a family of  sub-schemes of $\mathbb A^n$ 
parametrized by the affine space~$\mathbb A^m$ and let  
$\Phi: \mathcal{F} \To \mathbb A^m$ be a dominant morphism
which corresponds to an injective $K$-algebra 
homomorphism  ${\phi: K[\aaa] \To K[\aaa,\x]/I(\aaa,\x)}$. 
\end{assumption}

\begin{assumption}\label{section}
Let $P^{\mathstrut}=K[\x]$, 
let $i \in \{1,\dots, n\}$, let $c_{i+1}, \dots, c_n, \gamma \in K$,
and let $L=x_i-\ell\ $ be the linear polynomial with
$\ell =  \sum_{j> i} c_jx_j+\gamma$ if $i<n$ , and  $\ell = \gamma$ if $i=n$.
Moreover, we let $\x_{\hat{\imath}} = (x_1, \dots, x_{i-1}, x_{i+1}, \dots, x_n)$
and  identify $K[\aaa,\x]/(L)$ with $K[\aaa,\x_{\hat{\imath}}]$ 
via the isomorphism 
induced by  $\pi_L(x_i) = \ell$, $\pi_L(x_j) = x_j$ for $j \ne i$.
\end{assumption}

\begin{notation}\label{hyperGB}
The scheme $\Spec(K[\aaa, \x_{\hat{\imath}}]/\pi_L(I(\aaa, \x))$ is  
called the \textbf{$L$-hyperplane 
section of~$\mathcal{F}$} and denoted 
by $\mathcal{F}_L$. 
The morphism $\mathcal{F}_L \To \mathbb{A}^m$ 
which corresponds to the $K$-algebra homomorphism 
$\phi_L: K[\aaa] \To K[\aaa, \x_{\hat{\imath}}]/\pi_L(I(\aaa, x))$ 
canonically induced by $\phi$, is called $\Phi_L$.
Then let~$\sigma$ be a  
term ordering such that 
$x_1\!>_\sigma\! x_2 \!>_\sigma\!  \cdots \!>_\sigma\!  
x_n$, let $G_\sigma(\aaa,\x)=(g_1(\aaa,\x), \dots, g_s(\aaa,\x))$ 
be the universal 
reduced \hbox{$\sigma$-Gr\"obner} 
of~$\,\mathcal{F}$, and let $\hat{\sigma}$ be the term 
ordering induced by $\sigma$ on $\mathbb T(x_1,\dots, x_{i-1},x_{i+1},\dots, x_n)$.
\end{notation}

\begin{proposition}\label{hyperplaneFam}
Under Assumptions~\ref{familyOfSch} and~\ref{section}, 
let $G_\sigma(\aaa,\x)=\{g_1(\aaa,\x), \dots, g_s(\aaa,\x))\}$ 
be a monic \hbox{$\sigma$-Gr\"obner} of~$I(\aaa,\x)$, and
let~$\LT_\sigma(g_j(\aaa,\x))= \LT_{\hat{\sigma}}(\pi_L(g_j(\aaa,\x)))$ for every $j=1,\dots, s$. 
\begin{enumerate}

\item[(a)] The linear polynomial $L$ does not divide zero modulo $I(\aaa,\x)$.

\item[(b)] The set  $\{\pi_L(g_1(\aaa,\x)), \dots, \pi_L(g_s(\aaa,\x))\}$ is 
a minimal $\hat{\sigma}$-Gr\"obner basis of $\pi_L(I(\aaa,\x))$.

\item[(c)] If $L=x_i-\gamma$, then 
$\{\pi_L(g_1(\aaa,\x)), \dots, \pi_L(g_s(\aaa,\x))\}$ is 
the reduced $\hat{\sigma}$-Gr\"obner basis of $\pi_L(I(\aaa,\x))$.

\item[(d)] The $\aaa$-parameters are independent with respect 
to $\mathcal{F}_L$.
\end{enumerate}
\end{proposition}

\proof 
Claims (a), (b), (c) follow immediately from  Theorem~\ref{mainhyper}.
To prove claim (d) we observe that the $\aaa$-parameters are 
independent with respect to $\mathcal{F}$ by assumption. 
Therefore $G_\sigma(\aaa, \x) \ne \{1\}$ by Proposition~\ref{dependent}
and so $\LT_\sigma(I(\aaa,\x)) \ne (1)$. Our assumptions imply that 
also $\LT_{\hat{\sigma}}(I(\aaa,\x)_L) \ne(1)$ 
and hence the conclusion follows from Proposition~\ref{dependent}.
\endproof

The following example shows that  without the assumption of the proposition,
even if $\Phi$ is dominant, $\Phi_L$ needs not be such.

\begin{example}\label{not dominant}
Let $\mathcal{F}$ be the family of sub-schemes 
of $\mathbb{A}^4$ defined by 
${I(\aaa,\x) \!=\! (x^2\!-\!a_1y,\ y^2\!-\!a_2)}$. 
We check that $I(\aaa, \x)\cap K[\aaa] = (0)$,
so we conclude that the parameters are independent. 
However if $L = x-y$,  then $\mathcal{F}_L$ is defined
 by  $I(\aaa,\x)_L = (y^2-a_1y,\ y^2-a_2)$
and we have the following 
equality $I(\aaa, \x)_L\cap K[\aaa] = (a_1^2a_2-a_2^2)$
which means that the parameters with 
respect to  $\mathcal{F}_L$ are not independent anymore.
\end{example}

An easy  consequence of the proposition is that the 
non-constant coefficient list of~$G_{\hat{\sigma}}(\aaa,\x)_L$ 
is easily deduced from 
the non-constant coefficient list of $G_\sigma(\aaa, \x)$. Let us have a look at
an example which illustrates the proposition.

\begin{example}\label{coefflistsection}
Let $\mathcal{F}$ be the sub-scheme of $\mathbb A^7$ 
defined by the ideal $I(\aaa, \x)$ generated by the two polynomials
$$
F_1 = a_1xy-a_2y^2-w, \quad
F_2= a_2x^2+a_3y^2+z^2
$$
We pick a degree-compatible term ordering $\sigma$ with the property that
$x>_\sigma y>_\sigma z >_\sigma w$, and 
let $F_3=y^3+\frac{a_1^2}{a_2^3+a_1^2a_3}yz^2
+\frac{a_1a_2}{a_2^3+a_1^2a_3}xw+\frac{a_2^2}{a_2^3+a_1^2a_3}yw$.
Then $G_\sigma(\aaa,\x) = \{\frac{1}{a_1}F_1, \frac{1}{a_2}F_2, F_3\}$
is the universal reduced $\sigma$-Gr\"obner basis of~$\mathcal{F}$.
Therefore 
{\small
$$
{\rm NCC}_{G_\sigma(\aaa,\x)}=
\Big(-\frac{a_2}{a_1},\ -\frac{1}{a_1},\ \frac{a_3}{a_2}, \frac{1}{a_2},\ 
\frac{a_1^2}{a_2^3+a_1^2a_3},\  \frac{a_1a_2}{a_2^3+a_1^2a_3},\ 
 \frac{a_1^2}{a_2^3+a_1^2a_3}
\Big)
$$
}
The set of the leading terms of $G_\sigma(\aaa, \x)$ is $\{xy,\ x^2, \ y^3\}$ and
if we let $\ell = c_1w+c_2$ 
with $c_1, c_2 \in K$, $L = z-\ell$, then
claim (b) of the proposition implies that 
the substitution of $z$ with $\ell$ in 
$G_\sigma(\aaa, \x)$ produces a minimal 
$\hat{\sigma}$-Gr\"obner basis of~$\mathcal{F}_L$. 
For instance if $\ell = w-1$ we get  the equality 
$\pi_L(G_{\hat{\sigma}}(\aaa,\x))=\{\frac{1}{a_1}F_1,\  
\frac{1}{a_2}(a_2x^2+a_3y^2+(w-1)^2,\  \overline{F}_3\}$
where 
{\small
$$\overline{F}_3= 
y^3+\frac{a_1^2}{a_2^3+a_1^2a_3}y(w-1)^2
+\frac{a_1a_2}{a_2^3+a_1^2a_3}xw+\frac{a_2^2}{a_2^3+a_1^2a_3}yw
$$
$$
=y^3+\frac{a_1^2}{a_2^3+a_1^2a_3}yw^2
+\frac{a_1a_2}{a_2^3+a_1^2a_3}xw
+\frac{a_2^2-2a_1^2}{a_2^3+a_1^2a_3}yw
+ \frac{a_1^2}{a_2^3+a_1^2a_3} y
$$
}
It turns out that this is the reduced Gr\"obner basis, 
consequently we get the  equality
{\small
$${\rm NCC}_G = \Big(-\frac{a_2}{a_1}, \  -\frac{1}{a_1}, \ \frac{a_3}{a_2}, 
\ \frac{1}{a_2}, \ \frac{a_1^2}{a_2^3+a_1^2a_3}, \frac{a_1a_2}{a_2^3+a_1^2a_3},\ 
\frac{a_2^2-2a_1^2}{a_2^3+a_1^2a_3},\  \frac{a_1^2}{a_2^3+a_1^2a_3}
\Big)
$$
}
If we compute the elimination of $[x,y,z,w]$ from the 
ideal $(F_1,F_2)$ we get $(0)$, hence the parameters are independent. 
And if we compute the $\sigma$-scheme of $\mathcal{F}$ 
we get a scheme isomorphic to~$\mathbb{A}^3$.
\end{example}

\begin{remark}\label{similartheorems}
As we have seen, Proposition~\ref{hyperplaneFam}  is almost 
entirely based on Theorem~\ref {mainhyper}. Analogously, one can use Theorem~\ref{lifting} and Theorem~\ref{commonliftingGroebner} to deduce similar theorems in the case of families. This easy task is left to the reader.
\end{remark}
%
%
%
%
%
%

\subsection{The Hough Transform and its Dimension}

We recall the definition of Hough 
transform (see~\cite[Definition 3.11]{BR}).

\begin{definition}\label{Hough}{\bf (The Hough Transform)}\\
With the above notation and definitions, let
${p=(\xi_1, \dots, \xi_n)\in  \mathbb{A}^n}$.  
Then the scheme $\Gamma_{\!\aaa, p}$ 
 is said to be the {\bf Hough transform} of the point $p$ 
 with respect to the family $\Phi$.
If it is clear from the context, we simply say that the 
scheme $\Gamma_{\!\aaa, p}$ is
the {\bf Hough transform} of the point $p$ and we {\bf denote it by ${\rm H}_p$}.
We observe that if $p \notin {\rm Im}(\Psi)$, 
then ${\rm H}_p=\emptyset$.
\end{definition}

Hough transforms were invented by P.V.C. Hough (see~\cite{H}). Here we show an example which illustrates the original idea.

\begin{example}\label{line-points}
Let $\mathcal{F}$ be the hypersurface of $\mathbb{A}^4$ defined 
by the equation $x_2+a_1x_1+a_2=0$. It correspond to the $K$-algebra 
homomorphism $K[a_1,a_2] \To K[a_1,a_2,x_1,x_2]/(x_2+a_1x_1+a_2)$.
We have the following diagram
\begin{center}
\begin{tikzpicture}
  \node (F) {$\mathcal{F}$};
  \node (P) [below of=F, left of=F] {$\mathbb{A}^2_{(a_1, a_2)}$};
  \node (A) [right of=P, right of=P] {$\mathbb{A}^2_{(x_1,x_2)}$};
  \draw[->] (F) to node {$\quad ^\Psi$} (A);
  \draw[->] (F) to node {$\!\!\!\!\!^\Phi$} (P);
\end{tikzpicture}
\end{center}
\noindent
It is easy to check that $\dim(\mathcal{F})=3$ and that $\Phi$ and $\Psi$ are
dominant. It is clear that the Hough transform of
the point $(\xi_1, \xi_2)$ is the straight line in the parameter space defined by the equation $\xi_2+\xi_1a_1+a_2=0$. If some points, say  $p_1, p_2, \dots, p_s$, have Hough transforms which intersect in a point, say $(\alpha_1, \alpha_2)$, 
it means that the line 
$x_2+\alpha_1x_1+\alpha_2=0$ contains $p_1, p_2,\dots,p_s$.
Using this idea, Hough was able to detect line segments, and similarly arcs,
inside images.

\end{example}

Next, we show an  example where
$\Phi$ is dominant but $\Psi$ is not. 
\begin{example}\label{notdominant}
Let $\mathcal{F}$ be the sub-scheme of $\mathbb{A}^4$ defined 
by the ideal 
$$I=(x_1^2-x_1, \ \ x_1x_2-x_2, \  \ 
x_2^2+a_1a_2x_1-(a_1+a_2)x_2)$$
If we draw the diagram, it looks exactly the same as 
the diagram of Example~\ref{line-points}, 
but there are several differences.
It is easy to check that $\dim(\mathcal{F})=2$ and that $\Phi$ 
is dominant.
However, if we perform the elimination of $[a_1,a_2]$ we get 
the ideal  $(x_1^2-x_1, \ x_1x_2-x_2)$, which means that $\Psi$ 
is not dominant. 
In particular, the closure of the image of $\Psi$ is the union of the  
point $(0,0)$ and the line $x_1-1=0$ . We observe that the fiber 
of $\Psi$ over the point $(0,0)$ is the plane defined 
by $x_1=x_2=0$ while the fibers over the
points on $x=1$ are pairs of lines defined 
by $x_1=1,\ \ x_2 = c,\ \ (c-a_1)(c-a_2)=0$.
\end{example}

\medskip
The above example justifies the reason why in the next 
proposition we need to consider the image of $\, \Psi$.

\bigskip\goodbreak
\begin{proposition}\label{dimensionoffiber}
{\bf (Dimension of Hough Transforms)}\\
Let~$\mathbb{Y} \subseteq \mathbb{X}$ 
be an irreducible component of the closure 
of the image of $\, \Psi$,
let~$p$ be the generic point
of~$\mathbb{Y}$, and let 
$\X_{\bfalpha,\x}$ be the generic fiber of $\Phi$. 
\begin{enumerate}
\item[(a)]  
$\dim\, ({\rm H}_p) =  \dim\, (\mathcal{F})-  \dim\,(\mathbb{Y}) =
m+\dim(\X_{\bfalpha,\x})- \dim\,(\mathbb{Y})$.

\item[(b)]  If $\,\Psi$ is dominant and $\dim(\mathcal{F})=m$, 
then $\dim({\rm H}_p)=0$. 

\item[(c)] If $\dim({\rm H}_p)=0$ and the generators of $\,I$ are linear 
polynomials in the parameters $\aaa$, 
then~${\rm H}_p$ is a single rational point.
\end{enumerate}
\end{proposition}

\proof
In the proof we use the notation Kdim to indicate the Krull dimension.
To prove (a) we observe that 
we have the equality $\dim\,(\mathcal{F}) = {\rm Kdim}\, \big(K[\aaa,\x]/I(\aaa,\x)\big)$.
Then  we let~$\mathfrak{p}$ be the prime ideal which defines $\mathbb{Y}$
so that $\dim\,(\mathbb{Y}) = {\rm Kdim}\,(K[\aaa]/\mathfrak{p})$.
Since $\dim\, ({\rm H}_p)$ and $\dim\,(\X_{\bfalpha,\x})$ 
are the Krull dimensions of the fibers of $\Psi$ and $\Phi$ respectively, the claim 
follows from~\cite[Corollary 14.5]{E}.
To prove claim (b)  we observe that if $\,\Psi$ is dominant 
then $\mathbb{Y} =\mathbb{X} =\mathbb{A}^m$, 
hence $\dim(\mathbb{Y})=m$,
so we have $\dim({\rm H}_p)= m-m=0$.
Claim (c) follows from (b) and the fact that the coordinates of the 
points in ${\rm H}_p$ are the solutions of a linear system.
\endproof

Let us have a look at some examples.

\begin{example}\label{vertebral}
Let $\mathcal{F}$ be the sub-scheme of $\mathbb A^5$ 
defined by the ideal $I$ generated by the two polynomials
$$
F_1 = (x^2+y^2)^3-(a_1\,(x^2+y^2)-a_2\,(x^3-3xy^2))^2; 
\quad F_2 = a_1z-a_2x.
$$
If we pick a degree-compatible term ordering $\sigma$ such 
that  $z>_\sigma y>_\sigma x$, then $\LT_\sigma(F_1) = y^6$, 
${\LT_\sigma(F_2) = z}$ if $a_1\ne 0$,
and $\{F_1, \frac{1}{a_1}F_2\}$ is the reduced Gr\"obner basis 
of $I$. Using Proposition~\ref{flatness}, we 
get ${\mathcal{U}_\sigma = \mathbb{A}^2 \setminus\{a_1=0\}}$ 
and we see that 
 $\Phi^{-1}(\mathcal{U}_\sigma) \To \mathcal{U}_\sigma$ is free. 
 If we perform the elimination of $[a_1, a_2]$ we get the zero ideal, 
 hence also $\Psi$ is dominant, actually surjective.
 Counting dimensions as suggested by the proposition, we see that 
$ \dim\, (\Gamma_{\!\aaa, p}) = 0$ for the generic fiber. Since $a_1,a_2$ 
are quadratic and related by a linear equation,  the 
Hough transforms of the points in $\mathbb A^3$ are pairs of points. 
For instance,  if we pick the point $p = (1,1,1)$, its Hough transform 
is the pair of points $(\frac{1}{\sqrt{2}}, 1)$, 
$(-\frac{1}{\sqrt{2}}, -\frac{1}{\sqrt{2}})$.
\end{example}

\begin{example}\label{vertebral2}
We modify the above example in the following way.
Let $\mathcal{F}$ be the sub-scheme of~$\mathbb A^5$ defined by the 
ideal $I$ generated by the two polynomials
$$
F_1 = (x^2+y^2)^3-a_1\,((x^2+y^2)-(x^3-3xy^2))^2; 
\quad F_2 = z-a_2x.
$$
If we pick a degree-compatible term ordering $\sigma$ such 
that  $z>_\sigma y>_\sigma x$, then $\LT_\sigma(F_1) = y^6$, 
${\LT_\sigma(F_2) = z}$,
and $\{F_1, F_2\}$ is the reduced Gr\"obner basis 
of $I$. Using Proposition~\ref{flatness}  we see that  $\Phi$ is free. 
 If we perform the elimination of $[a_1, a_2]$ we get the zero ideal, 
 hence also $\Psi$ is dominant, actually surjective.
 Counting dimensions as suggested by the proposition, we see that 
$ \dim\, (\Gamma_{\!\aaa, p}) = 0$ for the generic fiber. 
Up to here the situation is similar to the above example. 
But now the two parameters $a_1, a_2$ are linear in the 
polynomials $F_1, F_2$, hence 
the Hough transforms of the generic point in $\mathbb A^3$ is a 
single point as described in the proposition.
It has coordinates $\alpha_1, \alpha_2$ where
$$\alpha_1 = 
\frac{(x^2+y^2)^3}{((x^2+y^2)-(x^3-3xy^2))^2}  \qquad 
\alpha_2 =  \frac{z}{x}
$$

\end{example}

\subsection{Hyperplane Sections and Hough Transforms}
\label{HyperSecHough}
As we have seen in Examples~\ref{zitrus}, \ref{sectparams}, and
Remark~\ref{similartheorems}, there is a concrete possibility of reconstructing 
ideals from this hyperplane sections. In particular, it is interesting to be able to
reconstruct a surface from a set of planar curves obtained by slicing it.
Here we show an example which suggests how to do it.

\begin{example}\label{recHough}
Suppose we want to reconstruct a surface using  five images 
obtained by slicing it with the hyperplanes $z=0$, $z=1$, $z=-1$, $z=2$. Suppose that a priori we know that the images contain curves of the family
$x^3-a_1y^2+a_2x+a_3y+a_4=0$. 
Using the Hough transforms of the points of the image, 
we discover these curves. They are described by the ideals
$(z,\ x^3-y^2)$, $(z-1,\ x^3 -y^2 -x -y -1)$, $(z+1,\ x^3 -y^2 +x +y +1)$, 
$(z-2,\ x^3 -y^2 -2x -2y -2)$. We proceed as we suggested in 
Example~\ref{zitrus} and reconstruct the surface. 
Its equation is $x^3-xz-y^2-yz-z=0$.
\end{example}

Why could this reconstruction be important?
Suppose we have the images of several parallel sections 
of a human organ, which is exactly what happens
with various types of tomography. 
Then we try to identify the cross-sectional
curves using Hough transforms. 
Once we have the equations of these curves,
even for a small portion of the organ, we can try to reconstruct 
the equation of the whole surface using ideas  
outlined in the above example. 
This hot topic is under investigation.

\goodbreak


\addcontentsline{toc}{section}{\quad\ \ References}

\centerline{\bf References}

\bibliographystyle{elsarticle-num}

\end{document}